\documentclass[11pt]{article}
\usepackage[T1]{fontenc}
\usepackage{amsfonts}
\usepackage{amsmath}
\usepackage{amssymb}
\usepackage{amsthm}
\usepackage{bbm}
\usepackage{bm}
\usepackage{mathrsfs}
\usepackage{verbatim}
\usepackage{setspace}
\usepackage{color}
\usepackage{enumitem}

\theoremstyle{plain}
\newtheorem{theorem}{Theorem}[section]
\newtheorem{proposition}[theorem]{Proposition}
\newtheorem{lemma}[theorem]{Lemma}

\theoremstyle{definition}

\newtheorem{remark}[theorem]{Remark}

\newtheorem{example}[theorem]{Example}

\newtheorem{assumption}[theorem]{Assumption}
\theoremstyle{remark}

\renewenvironment{thebibliography}[1]{%
\begin{oldthebibliography}{#1}%
\setlength{\baselineskip}{.9em}
\linespread{1}%
\small
\setlength{\parskip}{0ex}%
\setlength{\itemsep}{.2em}%
}%
{%
\end{oldthebibliography}%
}
\newcommand{\q}{\quad}

\newcommand{\eps}{\varepsilon}

\newcommand{\N}{\mathbb{N}}

\newcommand{\Q}{\mathbb{Q}}
\newcommand{\R}{\mathbb{R}}

\newcommand{\cA}{\mathcal{A}}
\newcommand{\cB}{\mathcal{B}}

\newcommand{\cE}{\mathcal{E}}
\newcommand{\cF}{\mathcal{F}}
\newcommand{\cG}{\mathcal{G}}
\newcommand{\cH}{\mathcal{H}}

\newcommand{\cP}{\mathcal{P}}

\newcommand{\bD}{\mathbf{D}}

\newcommand{\fX}{\mathfrak{X}}

\DeclareMathOperator{\dist}{dist}

\DeclareMathOperator{\tr}{Tr}

\DeclareMathOperator{\esssup}{ess\, sup}

\newcommand{\as}{\mbox{-a.s.}}

\newcommand{\1}{\mathbf{1}}

\newcommand{\br}[1]{\langle #1 \rangle}

\newcommand{\bomega}{\bar{\omega}}

\newcommand{\tomega}{\tilde{\omega}}
\newcommand{\fPO}{\mathfrak{P}(\Omega)}
\newcommand{\fM}{\mathfrak{M}}
\newcommand{\fMa}{\mathfrak{M}_a}
\newcommand{\qfaq}{\quad\mbox{for all}\quad}

\numberwithin{equation}{section}

\usepackage[pdfborder={0 0 0}]{hyperref}
\hypersetup{
  urlcolor = black,
  pdfauthor = {Marcel Nutz, Ramon van Handel},
  pdfkeywords = {Sublinear expectation; G-expectation; random G-expectation; Time-consistency; Optional sampling; Dynamic programming; Analytic set},
  pdftitle = {Constructing Sublinear Expectations on Path Space},
  pdfsubject = {Constructing Sublinear Expectations on Path Space},
  pdfpagemode = UseNone
}

\begin{document}
\title{
\mbox{Constructing Sublinear Expectations on Path Space}
\date{January 21, 2013}
\author{
  Marcel Nutz%
  \thanks{
  Dept.\ of Mathematics, Columbia University, New York.
 mnutz@math.columbia.edu\newline
\indent~\,
  The work of MN is partially supported by NSF grant DMS-1208985.}
  \and
  Ramon van Handel%
  \thanks{
  Sherrerd Hall rm.\ 227, Princeton University, Princeton.
rvan@princeton.edu\newline
\indent~\,
  The work of RvH is partially supported by NSF grant DMS-1005575.}
  }
}
\maketitle \vspace{-1em}

\begin{abstract}
We provide a general construction of time-consistent sublinear expectations on the space of continuous paths.
It yields the existence of the conditional $G$-expectation of a Borel-measurable (rather than quasi-continuous) random variable, a generalization of the random $G$-expectation, and an optional sampling theorem that holds without exceptional set. Our results also shed light on the inherent limitations to
constructing sublinear expectations through aggregation.
\end{abstract}

\vspace{.9em}

{\small
\noindent \emph{Keywords} Sublinear expectation; $G$-expectation; random $G$-expectation; Time-consistency; Optional sampling; Dynamic programming; Analytic set

\noindent \emph{AMS 2000 Subject Classification}
93E20; %
60H30; %
91B30;   %
28A05  %
}\\

\section{Introduction}\label{se:intro}

We study sublinear expectations on the space $\Omega=C_0(\R_+,\R^d)$ of continuous paths. Taking the dual point of view, we are interested in mappings
\[
  \xi \mapsto \cE_0(\xi)=\sup_{P\in\cP}E^P[\xi],
\]
where $\xi$ is a random variable and $\cP$ is a set of probability measures, possibly non-dominated. In fact, any sublinear expectation with certain continuity properties is of this form (cf.\ \cite[Sect.\,4]{FollmerSchied.04}). Under appropriate assumptions on $\cP$, we would like to construct a conditional expectation $\cE_\tau(\xi)$ at any stopping time $\tau$ of the
the filtration $\{\cF_t\}$ generated by the canonical process $B$ and establish the tower property
\begin{equation}\label{eq:towerIntro}
  \cE_\sigma(\cE_\tau(\xi))=\cE_\sigma(\xi)\quad\mbox{for stopping times}\quad \sigma\leq \tau,
\end{equation}
a property also known as time-consistency in this context. While it is not clear \emph{a priori} what to call a conditional expectation, a sensible requirement for $\cE_\tau(\xi)$ is to satisfy
\begin{equation}\label{eq:aggreg}
  \cE_\tau(\xi)= \mathop{\esssup^P}_{P'\in \cP(\tau;P)} E^{P'}[\xi|\cF_\tau]\quad P\as \quad\mbox{for all } P\in\cP,
\end{equation}
where $\cP(\tau;P)=\{P'\in\cP:\, P'=P\mbox{ on }\cF_\tau\}$;
see the related representations in \cite{SonerTouziZhang.2010dual, SonerTouziZhang.2010bsde}. This determines $\cE_\tau(\xi)$ up to polar sets---the measures in $\cP$ may be mutually singular---and corresponds, under a fixed $P\in\cP$, to the representations that are well known from the theory of risk measures (e.g., \cite{FollmerSchied.04}).
However, it is far from clear that one can in fact construct a random
variable $\cE_\tau(\xi)$ such that the property~\eqref{eq:aggreg} holds;
this is the aggregation problem.
Severe restrictions are necessary to construct $\cE_\tau(\xi)$ directly by
gluing together the right hand sides in~\eqref{eq:aggreg}; cf.\
\cite{Cohen.11,SonerTouziZhang.2010aggreg}.  We shall use a
different starting point, which will lead both to a general construction
of the conditional expectations $\cE_\tau(\xi)$ (Theorem~\ref{th:dpp})
and to insight on the inherent limitations to the aggregation problem
\eqref{eq:aggreg} (Section \ref{se:counterex}).

The main examples we have in mind are related to volatility uncertainty, where each $P\in\cP$ corresponds to a possible scenario for the volatility $d\br{B}_t/dt$. Namely, we shall consider the $G$-expectation~\cite{Peng.07, Peng.08} and its generalization to the ``random $G$''-expectation~\cite{Nutz.10Gexp}, where the range of possible volatilities is described by a random set $\bD$.
However, our general construction is much more broadly applicable;
for example, value functions of standard control problems
(under a given probability measure) can often be seen as sublinear
expectations on $\Omega$ by a push-forward, that is, by taking $\xi$ to be the reward functional and $\cP$ the set of possible laws of the controlled process (e.g., \cite{Nutz.11, Peng.04}).

Our starting point is a family of sets $\cP(\tau,\omega)$ of probability measures, where $\tau$ is a stopping time and $\omega\in\Omega$, satisfying suitable properties of measurability, invariance, and stability under pasting (Assumption~\ref{as:invarianceAndPasting}).
Roughly speaking, $\cP(\tau,\omega)$ represents all possible conditional laws of the
increments of the canonical process after time $\tau(\omega)$.
Taking inspiration from~\cite{SonerTouziZhang.2010dual}, we then define
\[
  \cE_\tau(\xi)(\omega):=\sup_{P\in\cP(\tau,\omega)} E^P[\xi^{\tau,\omega}],\quad\omega\in\Omega
\]
with $\xi^{\tau,\omega}(\omega'):=\xi(\omega\otimes_\tau\omega')$,
where $\omega\otimes_\tau\omega'$ denotes the path that equals
$\omega$ up to time $\tau(\omega)$ and whose increments after time
$\tau(\omega)$ coincide with $\omega'$.
Thus, $\cE_\tau(\xi)$ is defined for every single $\omega\in\Omega$, for any Borel-measurable (or, more generally, upper semianalytic) random variable $\xi$.
While $\cE_\tau(\xi)$ need not be Borel-measurable in general, we
show using the classical theory of analytic sets that $\cE_\tau(\xi)$
is always upper semianalytic (and therefore \emph{a fortiori}
universally measurable), and that it satisfies the
requirement~\eqref{eq:aggreg} and the tower
property~\eqref{eq:towerIntro}; cf.\ Theorem~\ref{th:dpp}.
We then show that our general result applies in the settings of
$G$-expectations and random $G$-expectations (Sections~\ref{se:Gexp} and
\ref{se:randomG}). Finally, we demonstrate that even in the fairly
regular setting of $G$-expectations, it is indeed necessary to consider
semianalytic functions: the conditional expectation of a Borel-measurable
random variable $\xi$ need not be Borel-measurable, even modulo a
polar set (Section~\ref{se:counterex}).

To compare our results with the previous literature, let us recall that
the $G$-expectation has been studied essentially with three different
methods: limits of PDEs \cite{Peng.07, Peng.08, Peng.10}, capacity theory
\cite{DenisHuPeng.2010, DenisMartini.06}, and the stochastic control
method of \cite{SonerTouziZhang.2010dual}. All these works start with very
regular functions $\xi$ and end up with random variables that are
quasi-continuous and results that hold up to polar sets (a random
variable is called quasi-continuous if it satisfies the Lusin property
uniformly in $P\in\cP$; cf.\ \cite{DenisHuPeng.2010}). Stopping times,
which tend to be discontinuous functions of $\omega$, could not be treated
directly (see \cite{LiPeng.09, NutzSoner.10, Song.11} for related partial
results) and the existence of conditional $G$-expectations beyond
quasi-continuous random variables remained open. We recall that not all
Borel-measurable random variables are quasi-continuous: for example, the
main object under consideration, the volatility of the canonical process,
is not quasi-continuous \cite{Song.12}. Moreover, even given a
quasi-continuous random variable $\xi$ and a closed set $C$, the indicator function
of $\{\xi\in C\}$ need not be quasi-continuous (cf.\ Section~\ref{se:counterex}), so that conditional
``$G$-probabilities'' are outside the scope of previous constructions.

The approach in the present paper is purely measure-theoretic and allows to treat general random variables and stopping times. Likewise, we can construct random $G$-expectations when $\bD$ is merely measurable, rather than satisfying an ad-hoc continuity condition as in~\cite{Nutz.10Gexp}; this is important since that condition did not allow to specify $\bD$ directly in terms of the observed historical volatility. Moreover, our method yields results that are more precise, in that they hold for every $\omega$ and not up to polar sets. In particular, this allows us to easily conclude that $\cE_\tau(\xi)$ coincides with the process $t\mapsto \cE_t(\xi)$ sampled at $\tau$, so that~\eqref{eq:towerIntro} may be seen as the optional sampling theorem for that nonlinear martingale (see \cite{NutzSoner.10} for a related partial result).

\section{General Construction}\label{se:main}

\subsection{Notation}

Let us start by cautioning the reader that our notation differs from the one in some related works in that we shall be shifting paths rather than the related function spaces. This change is necessitated  by our treatment of stopping times.

Let $\Omega=C_0(\R_+,\R^d)$ be the space of continuous paths $\omega=(\omega_u)_{u\geq0}$ in $\R^d$ with $\omega_0=0$ (throughout this section, $\R^d$ can be replaced by a
separable Fr\'echet space). We equip $\Omega$ with the topology of locally uniform convergence and denote by $\cF$ its Borel $\sigma$-field. Moreover, we denote by $B=\{B_u(\omega)\}$ the canonical process and by $(\cF_u)_{u\geq0}$ the (raw) filtration generated by $B$. Furthermore, let $\fPO$ be the set of all probability measures on $\Omega$, equipped with the topology of weak convergence; i.e., the weak topology induced by the bounded continuous functions on $\Omega$.
For brevity, ``stopping time'' will refer to a \emph{finite} (i.e.,
$[0,\infty)$-valued) $(\cF_u)$-stopping time throughout this paper.
We shall use various classical facts about processes on canonical spaces (see \cite[Nos.\,IV.94--103, pp.\,145--152]{DellacherieMeyer.78} for related background); in particular,
Galmarino's test: An $\cF$-measurable function $\tau: \Omega\to \R_+$ is a
stopping time if and only if $\tau(\omega)\leq t$ and
$\omega|_{[0,t]}=\omega'|_{[0,t]}$ imply $\tau(\omega)=\tau(\omega')$. Moreover,
given a stopping time $\tau$, an $\cF$-measurable function
$f$ is $\cF_\tau$-measurable if and only if
$f=f\circ\iota_\tau$, where $\iota_\tau:\Omega\to\Omega$ is the
stopping map $(\iota_\tau(\omega))_t=\omega_{t\wedge\tau(\omega)}$.

Let $\tau$ be a stopping time. The concatenation of $\omega, \tomega\in \Omega$ at $\tau$ is the path
\[
 (\omega\otimes_\tau \tomega)_u := \omega_u \1_{[0,\tau(\omega))}(u) + \big(\omega_{\tau(\omega)} + \tomega_{u-\tau(\omega)}\big) \1_{[\tau(\omega), \infty)}(u),\quad u\geq 0.
\]
Given a function $\xi$ on $\Omega$ and $\omega\in\Omega$,
we define the function $\xi^{\tau,\omega}$ on $\Omega$ by
\[
  \xi^{\tau,\omega}(\tomega) :=\xi(\omega\otimes_\tau \tomega),\q \tomega\in\Omega.
\]
We note that $\omega\mapsto \xi^{\tau,\omega}$ depends only on $\omega$
up to time $\tau(\omega)$;
that is, if $\omega=\omega'$ on $[0,\tau(\omega)]$, then $\xi^{\tau,\omega}=\xi^{\tau,\omega'}$
(and $\tau(\omega)=\tau(\omega')$ by Galmarino's test).
Let $\sigma$ be another stopping time such that $\sigma\leq\tau$ and let $\omega\in\Omega$. Then
\[
  \theta:=(\tau-\sigma)^{\sigma,\omega}=\tau(\omega\otimes_\sigma \cdot )- \sigma(\omega)
\]
is again a stopping time; indeed, with $s:=\sigma(\omega)$, we have
\[
  \{\theta\leq t\}=\{\tau(\omega\otimes_s \cdot) \leq t+s\}\in \cF_{t+s-s}=\cF_t,\quad t\geq0.
\]

For any probability measure $P\in\fPO$, there is a regular conditional
probability distribution $\{P^\omega_\tau\}_{\omega\in\Omega}$
given $\cF_\tau$.
That is, $P^\omega_\tau\in\fPO$ for each $\omega$, while $\omega\mapsto
P^\omega_\tau(A)$ is $\cF_\tau$-measurable for any $A\in\cF$ and
\[
  E^{P^\omega_\tau}[\xi]=E^P[\xi|\cF_\tau](\omega)\quad \mbox{for}\quad P\mbox{-a.e.}\;\omega\in\Omega
\]
whenever $\xi$ is $\cF$-measurable and bounded.
Moreover, $P^\omega_\tau$ can be chosen to be concentrated on the
set of paths that coincide with $\omega$ up to time $\tau(\omega)$,
\begin{equation*}%
  P^\omega_\tau\big\{\omega'\in \Omega: \omega' = \omega \mbox{ on } [0,\tau(\omega)]\big\} = 1
  \quad \mbox{for all }\omega\in\Omega;
\end{equation*}
cf.\ \cite[p.\,34]{StroockVaradhan.79}.
We define the probability measure $P^{\tau,\omega}\in \fPO$ by
\[
  P^{\tau,\omega}(A):=P^\omega_\tau(\omega\otimes_\tau A),\quad A\in \cF, \quad\mbox{where }\omega\otimes_\tau A:=\{\omega\otimes_\tau \tomega:\, \tomega\in A\}.
\]
We then have the identities
\[
  E^{P^{\tau,\omega}}[\xi^{\tau,\omega}]=E^{P^\omega_\tau}[\xi] =E^P[\xi|\cF_\tau](\omega) \quad \mbox{for}\quad P\mbox{-a.e.}\;\omega\in\Omega.
\]

To avoid cumbersome notation, it will be useful to define integrals for all measurable functions $\xi$ with values in the extended real line $\overline{\R}=[-\infty,\infty]$. Namely, we set
\[
  E^P[\xi]:=E^P[\xi^+]-E^P[\xi^-]
\]
if $E^P[\xi^+]$ or $E^P[\xi^-]$ is finite, and we use the convention
\[
 E^P[\xi]:=-\infty \quad\mbox{if}\quad E^P[\xi^+]=E^P[\xi^-]=+\infty.
\]
The corresponding convention is used for the conditional expectation with respect to a $\sigma$-field $\cG\subseteq \cF$; that is, $E^P[\xi|\cG]=E^P[\xi^+|\cG]-E^P[\xi^-|\cG]$ $P$-a.s.\ on the set where $E^P[\xi^+|\cG]$ or $E^P[\xi^-|\cG]$ is finite, and $E^P[\xi|\cG]=-\infty$ on the complement.

Next, we recall some basic definitions from the theory of analytic sets; we refer to
\cite[Ch.\,7]{BertsekasShreve.78} or \cite[Ch.\,8]{Cohn.80}
for further background. A subset of a Polish space is called analytic if it is the image of a Borel subset of another Polish space under a Borel-measurable mapping.
In particular, any Borel set is analytic. The collection of analytic sets is stable under countable intersections and unions, but in general not under complementation. The $\sigma$-field $\cA$ generated by the analytic sets is called the analytic $\sigma$-field and $\cA$-measurable functions are called analytically measurable.
Moreover, given a $\sigma$-field $\cG$ on any set, the universal completion of $\cG$ is the $\sigma$-field $\cG^*=\cap_P \cG^P$, where $P$ ranges over all probability measures on $\cG$ and $\cG^P$ is the completion of $\cG$ under $P$. If $\cG$ is the Borel $\sigma$-field of a Polish space, we have the inclusions
\[
  \cG\,\subseteq\, \cA \,\subseteq\, \cG^* \,\subseteq\, \cG^P
\]
for any probability measure $P$ on $\cG$. Finally, an $\overline{\R}$-valued function $f$ is called upper semianalytic if $\{f>c\}$ (or equivalently $\{f\geq c\}$) is analytic for each $c\in\R$.
In particular, any Borel-measurable function is upper semianalytic, and any upper semianalytic function is analytically and universally measurable.

Finally, note that since $\Omega$ is a Polish space, $\fPO$ is again
a Polish space \cite[Prop.\,7.20, p.\,127 and Prop.\,7.23,
p.\,131]{BertsekasShreve.78}, and so is the product $\fPO\times\Omega$.

\subsection{Main Result}

For each $(s,\omega)\in \R_+\times\Omega$, we fix a set $\cP(s,\omega)\subseteq \fPO$. We assume that these sets are adapted in that
\[
  \cP(s,\omega)=\cP(s,\tomega)\quad\mbox{if}\quad \omega|_{[0,s]}=\tomega|_{[0,s]}.
\]
In particular, the set $\cP(0,\omega)$ is independent of $\omega$ (since
all paths start at zero) and we shall denote it by $\cP$. We assume
throughout that $\cP\neq\varnothing$. If $\sigma$ is a stopping time, we
set
\[
  \cP(\sigma,\omega):=\cP(\sigma(\omega),\omega).
\]
The following are the conditions for our main result.

\begin{assumption}\label{as:invarianceAndPasting}
  Let $s\in\R_+$, let $\tau$ be a stopping time such  that $\tau\geq s$, let $\bomega\in\Omega$ and $P\in \cP(s,\bomega)$. Set $\theta:=\tau^{s,\bomega}-s$.
  \begin{enumerate}[topsep=3pt, partopsep=0pt, itemsep=1pt,parsep=2pt]
    \item \emph{Measurability:} The graph $\{(P',\omega): \omega\in\Omega,\; P'\in \cP(\tau,\omega)\} \,\subseteq\, \fPO\times\Omega$ is analytic.

    \item \emph{Invariance:} We have $P^{\theta,\omega} \in\cP(\tau,\bomega\otimes_s\omega)$ for $P$-a.e.\ $\omega\in\Omega$.

    \item \emph{Stability under pasting:} If $\nu: \Omega \to \fPO$ is an $\cF_\theta$-measurable kernel and $\nu(\omega)\in \cP(\tau,\bomega\otimes_s\omega)$ for $P$-a.e.\ $\omega\in\Omega$,
    then the measure defined by
    \begin{equation}\label{eq:defPbar}
      \bar{P}(A)=\iint (\1_A)^{\theta,\omega}(\omega') \,\nu(d\omega';\omega)\,P(d\omega),\quad A\in \cF
    \end{equation}
    is an element of $\cP(s,\bomega)$.
  \end{enumerate}
\end{assumption}

\begin{remark}\label{rk:conseqOfAssumpt}
  \begin{enumerate}[topsep=3pt, partopsep=0pt, itemsep=1pt,parsep=2pt]
    \item[(a)] As $\cP$ is nonempty, Assumption~(ii) implies that the set
$\{\omega\in\Omega:\, \cP(\tau,\omega)=\varnothing\}$ is $P$-null for any
$P\in\cP$ and stopping time $\tau$.
    \item[(b)] At an intuitive level, Assumptions~(ii) and (iii)
suggest the identity
$\cP(\tau,\omega)=\{P^{\tau,\omega}:\, P\in\cP\}$. This expression is not well-defined because $P^{\tau,\omega}$ is defined only up to a $P$-nullset; nevertheless, it sheds some light on the relations between the sets of measures that we have postulated.
  \end{enumerate}
\end{remark}

The following is the main result of this section. We denote by $\esssup^P$
the essential supremum under $P\in\fPO$ and use the convention
$\sup\varnothing = -\infty$.

\begin{theorem}\label{th:dpp}
   Let Assumption~\ref{as:invarianceAndPasting} hold true, let $\sigma\leq\tau$ be stopping times and let $\xi:\Omega\to\overline{\R}$ be an upper semianalytic function. Then the function
   \[
     \cE_\tau(\xi)(\omega):=\sup_{P\in\cP(\tau,\omega)} E^P[\xi^{\tau,\omega}],\quad\omega\in\Omega
   \]
   is $\cF_\tau^*$-measurable and upper semianalytic. Moreover,
   \begin{equation}\label{eq:DPP}
     \cE_\sigma(\xi)(\omega) = \cE_\sigma(\cE_\tau(\xi))(\omega)\quad\mbox{for all}\quad \omega\in\Omega.
   \end{equation}
   Furthermore,
   \begin{equation}\label{eq:esssupRep}
     \cE_\tau(\xi) = \mathop{\esssup^P}_{P'\in \cP(\tau;P)} E^{P'}[\xi|\cF_\tau]\quad P\as\qfaq P\in\cP,
   \end{equation}
   where $\cP(\tau;P)=\{P'\in \cP:\, P'=P \mbox{ on } \cF_\tau\}$, and in particular
   \begin{equation}\label{eq:esssupDPP}
     \cE_\sigma(\xi) = \mathop{\esssup^P}_{P'\in \cP(\sigma;P)} E^{P'}[\cE_\tau(\xi)|\cF_\sigma]\quad P\as\qfaq P\in\cP.
   \end{equation}
\end{theorem}

\begin{remark}
  \begin{enumerate}[topsep=3pt, partopsep=0pt, itemsep=1pt,parsep=2pt]
    \item It is immediate from our definitions that $\cE_\tau(\xi)$ coincides (at every $\omega$) with the process $\cE(\xi):\,(t,\omega)\mapsto \cE_t(\xi)(\omega)$ sampled at the stopping time $\tau$. That is, the (often difficult) problem of aggregating the family $\{\cE_\tau(\xi)\}_\tau$ into a process is actually trivial---the reason is that the definitions are made without exceptional sets. Thus, the semigroup property~\eqref{eq:DPP} amounts to an optional sampling theorem for the nonlinear martingale $\cE(\xi)$.

    \item If Assumption~\ref{as:invarianceAndPasting} holds for deterministic times instead of stopping times, then so does the theorem. This will be clear from the proof.

    \item Let $\xi$ be upper semianalytic and let $\xi'$ be another function such that $\xi=\xi'$ $P$-a.s.\ for all $P\in\cP$. Then
    $\cE_\tau(\xi) = \esssup^P_{P'\in \cP(\tau;P)} E^{P'}[\xi'|\cF_\tau]$  $P$-a.s.\ for all $P\in\cP$ by~\eqref{eq:esssupRep}. In particular, if $\xi'$ is upper semianalytic, we have $\cE_\tau(\xi)=\cE_\tau(\xi')$ $P$-a.s.\ for all $P\in\cP$.

    \item The basic properties of the sublinear expectation are evident
from the definition. In particular,
$\cE_\tau(\1_A\xi)(\omega)=\1_A(\omega)\cE_\tau(\xi)(\omega)$ if
$A\in\cF_\tau$ and $\cP(\tau,\omega)\neq\varnothing$. (The latter
restriction
could be omitted with the convention $0 (-\infty)=-\infty$, but this seems somewhat daring.)
  \end{enumerate}
\end{remark}
\pagebreak

\begin{proof}[Proof of Theorem~\ref{th:dpp}]
  For brevity, we set $V_\tau:=\cE_\tau(\xi)$.

  \medskip
  \noindent\emph{Step~1.}
  We start by establishing the measurability of $V_\tau$. To this end, let
  $\fX=\fPO\times \Omega$ and consider the mapping $K: \fX\to \fPO$ defined by
  \[
    K(A;P,\omega) = E^P[(\1_A)^{\tau,\omega}],\quad A\in \cF.
  \]
  Let us show that $K$ is a Borel kernel; i.e.,
  \[
    K: \fX\to \fPO\quad\mbox{is Borel-measurable.}
  \]
  This is equivalent to saying that $(P,\omega)\mapsto E^P[f^{\tau,\omega}]$ is Borel-measurable
  whenever $f:\Omega\to \R$ is bounded and Borel-measurable (cf.\ \cite[Prop.\,7.26, p.\,134]{BertsekasShreve.78}).
  To see this, consider more generally the set $W$ of all bounded Borel functions $g: \Omega\times\Omega\to\R$ such that
  \begin{equation}\label{eq:KborelAux}
    (P,\omega)\mapsto E^P[g(\omega,\cdot)] \quad \mbox{is Borel-measurable.}
  \end{equation}
  Then $W$ is a linear space and if $g_n\in W$ increase to a bounded function $g$, then~\eqref{eq:KborelAux} is satisfied as $(P,\omega)\mapsto E^P[g(\omega,\cdot)]$ is the pointwise limit of the Borel-measurable  functions $(P,\omega)\mapsto E^P[g_n(\omega,\cdot)]$. Moreover, $W$ contains any bounded, uniformly continuous function $g$. Indeed,
  if $\rho$ is a modulus of continuity for $g$ and $(P^n,\omega^n)\to(P,\omega)$ in $\fX$, then
  \begin{align*}
    \big|E^{P^n}[g(\omega_n,\cdot)&]-E^P[g(\omega,\cdot)] \big| \\
    &\leq \big|E^{P^n}[g(\omega_n,\cdot)]-E^{P^n}[g(\omega,\cdot)] \big| + \big|E^{P^n}[g(\omega,\cdot)]-E^P[g(\omega,\cdot)] \big| \\
    &\leq \rho(\dist(\omega^n,\omega)) + \big|E^{P^n}[g(\omega,\cdot)]-E^P[g(\omega,\cdot)] \big| \to 0,
  \end{align*}
  showing that $(P,\omega)\mapsto E^P[g(\omega,\cdot)]$ is continuous and thus Borel-measurable. Since the uniformly continuous functions generate the Borel $\sigma$-field on $\Omega\times\Omega$, the monotone class theorem implies that $W$ contains all bounded Borel-measurable functions and in particular the function $(\omega,\omega')\mapsto f^{\tau,\omega}(\omega')$. Therefore, $K$ is a Borel kernel.

  It is a general fact that Borel kernels integrate upper semianalytic functions into upper semianalytic ones (cf.\ \cite[Prop.\,7.48, p.\,180]{BertsekasShreve.78}). In particular,
  as $\xi$ is upper semianalytic, the function
  \[
    (P,\omega)\mapsto E^P[\xi^{\tau,\omega}]\equiv\int \xi(\omega') K(d\omega';P,\omega)
  \]
  is upper semianalytic.
In conjunction with Assumption~\ref{as:invarianceAndPasting}(i), which states that $\cP(\tau,\omega)$ is the $\omega$-section of an analytic subset of $\fPO\times\Omega$, a variant of the projection theorem (cf.\ \cite[Prop.\,7.47, p.\,179]{BertsekasShreve.78}) allows us to conclude that
  \[
    \omega\mapsto V_\tau(\omega)=\sup_{P\in\cP(\tau,\omega)} E^P[\xi^{\tau,\omega}]
  \]
  is again upper semianalytic as a function on $\Omega$.
  It remains to show that $V_\tau$ is measurable with respect to the universal completion $\cF_\tau^*$.
As $\omega\mapsto V_\tau(\omega)$ depends only on
$\omega$ up to time $\tau(\omega)$, this follows directly from the
following universally measurable extension of Galmarino's test.

\begin{lemma}
\label{lem:univgalmarino}
  Let $X:\Omega\to\overline{\R}$ be $\cF^*$-measurable and let $\tau$
  be a stopping time.  Then $X$ is $\cF_\tau^*$-measurable if and only if
  $X(\omega)=X(\iota_\tau(\omega))$ for all $\omega\in\Omega$, where
  $\iota_\tau:\Omega\to\Omega$ is the stopping map
  $(\iota_\tau(\omega))_t=\omega_{t\wedge\tau(\omega)}$.
\end{lemma}

\begin{proof}
  By Galmarino's test, the stopping map $\iota_\tau$ is measurable
  from $(\Omega,\cF_\tau)$ to $(\Omega,\mathcal{F})$. As a consequence, $\iota_\tau$ is also measurable from
  $(\Omega,\cF_\tau^*)$ to $(\Omega,\cF^*)$; cf.\ \cite[Lem.\,8.4.6,
  p.\,282]{Cohn.80}. Hence, if
  $X=X\circ\iota_\tau$, then $X$ is $\cF_\tau^*$-measurable.

  To see the converse, recall that if $Y$ is $\cF^*_\tau$ measurable and $P\in\fPO$, there exists an $\cF_\tau$ measurable $Y'$ such that $Y'=Y$ $P$-a.s. Suppose that there exists $\omega\in\Omega$ such that
  $X(\omega)\ne X(\iota_\tau(\omega))$. Let $P$ be the probability measure
  that puts mass $1/2$ on $\omega$ and $\iota_\tau(\omega)$, and let
  $X'$ be any random variable such that $X'=X$ $P$-a.s.  Then clearly
  $X'(\omega)\ne X'(\iota_\tau(\omega))$, so that $X'$ is not
  $\cF_\tau$-measurable by Galmarino's test.  It follows that $X$ is not $\cF_\tau^*$-measurable.
\end{proof}

We now collect some basic facts about composition of upper semianalytic
random variables that will be used in the sequel without further comment.

\begin{lemma}\label{lem:usacompose}
  Let $\xi:\Omega\to\overline{\R}$ be upper semianalytic,
  let $\tau$ be a stopping time, and let
  $\nu:\Omega\to\fPO$ be a Borel-measurable kernel.  Then
  \begin{enumerate}[topsep=3pt, partopsep=0pt, itemsep=1pt,parsep=2pt]
    \item $\xi^{\tau,\omega}$ is upper semianalytic for every $\omega\in\Omega$;
    \item $\omega\mapsto E^{\nu(\omega)}[\xi^{\tau,\omega}]$ is upper semianalytic.
  \end{enumerate}
\end{lemma}

\begin{proof}
  If $X$ is upper semianalytic and $\iota$ is Borel-measurable,
  then $X\circ\iota$ is upper semianalytic
  \cite[Lem.\,7.30, p.\,178]{BertsekasShreve.78}.
  The first statement now follows immediately as
  $\xi^{\tau,\omega}=\xi\circ\iota$ with $\iota(\omega')=\omega\otimes_\tau
  \omega'$.  For the second statement, note that we have shown above
  that $(P,\omega)\mapsto E^P[\xi^{\tau,\omega}]$ is upper semianalytic,
  while $\omega\mapsto (\nu(\omega),\omega)$ is Borel-measurable by
  assumption.
\end{proof}

We also recall for future reference that the composition of
two universally measurable functions is again universally measurable
\cite[Prop.\,7.44, p.\,172]{BertsekasShreve.78}.

  \medskip
  \noindent\emph{Step~2.} We turn to the proof of~\eqref{eq:DPP}, which we can cast as
  \begin{equation}\label{eq:DPPres}
    \sup_{P\in\cP(\sigma,\bomega)} E^P\big[ \xi^{\sigma,\bomega} \big] = \sup_{P\in\cP(\sigma,\bomega)} E^P\big[ V_\tau^{\sigma,\bomega} \big]\quad \mbox{for all}\quad  \bomega\in\Omega,
  \end{equation}
  where $V_\tau^{\sigma,\bomega}:=(V_\tau)^{\sigma,\bomega}$.
  In the following, we fix $\bomega\in \Omega$, and for brevity, we set
  \[
      s:=\sigma(\bomega)\quad\mbox{and}\quad\theta:=(\tau-\sigma)^{\sigma,\bomega}\equiv\tau(\bomega\otimes_s\cdot)-s.
  \]

First, let us prove the inequality ``$\leq$'' in~\eqref{eq:DPPres}.
  Fix $P\in \cP(\sigma,\bomega)\equiv\cP(s,\bomega)$.
Assumption~\ref{as:invarianceAndPasting}(ii) shows that $P^{\theta,\omega}\in \cP(\tau,\bomega\otimes_s \omega)$ for $P$-a.e.\ $\omega\in \Omega$ and hence
  \begin{align*}
    E^{P^{\theta,\omega}}\big[(\xi^{s,\bomega})^{\theta,\omega}\big]
    & = E^{P^{\theta,\omega}}\big[\xi^{\theta(\omega)+s,\bomega \otimes_s\omega}\big]\\
    & = E^{P^{\theta,\omega}}\big[\xi^{\tau,\bomega \otimes_s\omega}\big]\\
    & \leq \sup_{P'\in\cP(\tau,\bomega\otimes_s \omega)} E^{P'}\big[\xi^{\tau,\bomega \otimes_s\omega}\big]\\
    & = V_\tau^{s,\bomega}(\omega)\quad\mbox{for $P$-a.e.\ $\omega\in \Omega$.}
  \end{align*}
  Taking $P(d\omega)$-expectations on both sides, we obtain that
  \[
    E^P\big[\xi^{s,\bomega}\big]
    \leq E^P\big[V_\tau^{s,\bomega}\big].
  \]
  The inequality ``$\leq$'' in~\eqref{eq:DPPres} follows by taking the supremum over $P\in\cP(s,\bomega)$.

  We now show the converse inequality ``$\geq$'' in~\eqref{eq:DPPres}.
Fix $\varepsilon>0$.
We begin by noting that since the sets $\cP(\tau,\omega)$
are the $\omega$-sections of an analytic set in $\fPO\times\Omega$, the
Jankov-von Neumann theorem in the form of \cite[Prop.\,7.50,
p.\,184]{BertsekasShreve.78} yields a universally
measurable function $\tilde\nu:\Omega\to\fPO$ such that
  \[
    E^{\tilde{\nu}(\omega)} [\xi^{\tau,\omega}] \geq
    \begin{cases}
      V_\tau(\omega)-\eps & \text{if } V_\tau(\omega)<\infty\\
      \eps^{-1} & \text{if }V_\tau(\omega)=\infty
    \end{cases}
  \]
and $\tilde\nu(\omega)\in\cP(\tau,\omega)$ for all $\omega\in\Omega$
such that $\cP(\tau,\omega)\ne\varnothing$.

Fix $P\in\cP(s,\bomega)$.  As the composition of universally
measurable functions is universally measurable, the map
$\omega\mapsto\tilde\nu(\bomega\otimes_s\iota_\theta(\omega))$ is
$\cF_\theta^*$-measurable by Lemma \ref{lem:univgalmarino}.
Therefore, there exists an $\cF_\theta$-measurable kernel
$\nu:\Omega\to\fPO$ such that
$\nu(\omega)=\tilde\nu(\bomega\otimes_s\iota_\theta(\omega))$ for
$P$-a.e.\ $\omega\in\Omega$.  Moreover,
Assumption~\ref{as:invarianceAndPasting}(ii) shows that
$\cP(\tau,\bomega \otimes_s\omega)$ contains the element
$P^{\theta,\omega}$ for $P$-a.e.\ $\omega\in\Omega$, so that
$\{\omega\in\Omega:\cP(\tau,\bomega \otimes_s\omega)\ne\varnothing\}$
has full $P$-measure.  Thus
  \begin{equation}\label{eq:epsOptimal}
    \nu(\cdot)\in \cP(\tau,\bomega\otimes_s\cdot)\quad\!\mbox{and}\quad\!\! E^{\nu(\cdot)} [\xi^{\tau,\bomega\otimes_s\cdot}] \geq     \begin{cases}
      V_\tau^{s,\bomega}-\eps & \!\text{on } \{V_\tau^{s,\bomega}<\infty\}\\
      \eps^{-1} & \!\text{on }\{V_\tau^{s,\bomega}=\infty\}
    \end{cases}
    \;\,P\as
  \end{equation}
  Let $\bar{P}$ be the measure defined by
  \begin{equation}\label{eq:defBarPinProof}
    \bar{P}(A)=\iint (\1_A)^{\theta,\omega}(\omega') \,\nu(d\omega';\omega)\,P(d\omega),\quad A\in \cF;
  \end{equation}
  then $\bar{P}\in \cP(s,\bomega)$ by Assumption~\ref{as:invarianceAndPasting}(iii). In view of~\eqref{eq:epsOptimal}, we
  conclude that
  \begin{align*}
    E^P\big[ V_\tau^{s,\bomega} \wedge \eps^{-1} \big]
      & \leq E^P\big[ E^{\nu(\cdot)} [\xi^{\tau,\bomega\otimes_s\cdot}]\big] + \eps \\
      & = E^P\big[ E^{\nu(\cdot)} [(\xi^{s,\bomega})^{\theta,\cdot}]\big] + \eps \\
      & = E^{\bar{P}}[\xi^{s,\bomega}] + \eps \\
      &\leq \sup_{P'\in\cP(s,\bomega)} E^{P'}\big[ \xi^{s,\bomega} \big] + \eps.
  \end{align*}
  As $\eps>0$ and $P\in\cP(s,\bomega)$ were arbitrary, this completes the proof of~\eqref{eq:DPPres}.

\medskip

Before continuing with the proof, we record a direct consequence of
disintegration of measures for ease of reference.  Its proof is omitted.

\begin{lemma}\label{le:rcpdOfBarP}
  In the setting of Assumption~\ref{as:invarianceAndPasting}(iii), we have
  \[
  \bar{P}^{\theta,\omega}=\nu(\omega) \quad \mbox{for}\quad \bar{P}\mbox{-a.e.} \mbox{ and } P\mbox{-a.e.}\; \omega\in \Omega.
  \]
\end{lemma}

We return to the proof of the theorem.

  \medskip
  \noindent\emph{Step~3.}
  Fix $P\in\cP$; we show the representation~\eqref{eq:esssupRep}. Let $P'\in\cP(\tau;P)$; then ${P'}^{\tau,\omega}\in \cP(\tau,\omega)$ $P'$-a.s.\ by
  Assumption~\ref{as:invarianceAndPasting}(ii) and hence
  \[
    V_\tau=\sup_{P''\in\cP(\tau,\omega)} E^{P''}[\xi^{\tau,\omega}] \geq E^{{P'}^{\tau,\omega}}[\xi^{\tau,\omega}]=E^{P'}[\xi|\cF_\tau](\omega)\quad\mbox{for}\quad {P'}\mbox{-a.e.}\,\omega\in\Omega.
  \]
  Both sides of this inequality are $\cF_\tau^*$-measurable. Moreover, we have $P=P'$ on $\cF_\tau$, and since measures extend uniquely to the universal completion, we also have $P=P'$ on $\cF_\tau^*$. Therefore, the inequality holds also $P$-a.s.
  Since $P'\in \cP(\tau;P)$ was arbitrary, we conclude that
  \[
    V_\tau \geq \mathop{\esssup^P}_{P'\in \cP(\tau;P)} E^{P'}[\xi|\cF_\tau]\quad P\as
  \]
  It remains to show the converse inequality. Let $\eps>0$ and consider the construction in Step~2 for the special case $s=0$ (in which there is no dependence on $\bomega$). Then the measure $\bar{P}$ from~\eqref{eq:defBarPinProof} is in $\cP$ by Assumption~\ref{as:invarianceAndPasting}(iii) and it coincides with $P$ on $\cF_\tau$; that is, $\bar{P}\in\cP(\tau;P)$. Using Lemma~\ref{le:rcpdOfBarP} and~\eqref{eq:epsOptimal}, we obtain that
  \[
    E^{\bar{P}}[\xi|\cF_\tau](\omega)=E^{\bar{P}^{\tau,\omega}}[\xi^{\tau,\omega}]=E^{\nu(\omega)}[\xi^{\tau,\omega}]\geq (V_\tau(\omega)-\eps)\wedge\eps^{-1}
  \]
  for $P$-a.e.\ $\omega\in\Omega$. Since $\eps>0$ was arbitrary, it follows that
  \[
    \mathop{\esssup^P}_{P'\in \cP(\tau;P)} E^{P'}[\xi|\cF_\tau] \geq V_\tau\quad P\as,
  \]
  which completes the proof of~\eqref{eq:esssupRep}.

  \medskip
  \noindent\emph{Step~4.}
  It remains to note that~\eqref{eq:DPP} and~\eqref{eq:esssupRep} applied to $V_\tau$ yield that
  \[
    \cE_\sigma(\xi)=\cE_\sigma(V_\tau)=\mathop{\esssup^P}_{P'\in \cP(\sigma;P)} E^{P'}[V_\tau|\cF_\sigma]\quad P\as\qfaq P\in\cP,
  \]
  which is~\eqref{eq:esssupDPP}. This completes the proof of Theorem~\ref{th:dpp}.
\end{proof}

\section{Application to $G$-Expectations}\label{se:Gexp}

We consider the set of local martingale measures
\[
  \fM=\big\{P\in\fPO:\, B \mbox{ is a local $P$-martingale}\big\}
\]
and its subset
\[
  \fMa=\big\{P\in\fM:\, \br{B}^P \mbox{ is absolutely continuous $P$-a.s.}\big\},
\]
where $\br{B}^P$ is the $\R^{d\times d}$-valued quadratic variation process of $B$ under $P$ and absolute continuity refers to the Lebesgue measure.
We fix a nonempty, convex and compact set
$\bD\subseteq \R^{d\times d}$ of matrices and consider the set
\[
  \cP_\bD=\big\{P\in\fMa:\, d\br{B}^P_t/dt \in \bD\; P\times dt\mbox{-a.e.}\big\}.
\]
We remark that defining $d\br{B}^P_t/dt$ up to nullsets, as required in the above formula, causes no difficulty because $\br{B}^P$ is \emph{a priori} absolutely continuous under $P$. A detailed discussion is given around~\eqref{eq:hata}, when we need a measurable version of this derivative. Moreover, we note that $\cP_\bD$ consists of true martingale measures because $\bD$ is bounded---the definition of $\fM$ is made in anticipation of the subsequent section.

It is well known that the sublinear expectation
\[
  \cE_0^\bD(\xi):=\sup_{P\in \cP_\bD} E^P[\xi]
\]
yields the $G$-expectation on the space $\mathbb{L}^1_G$ of quasi-continuous functions if $G:\R^{d\times d}\to\R$ is given by
\[
  G(\Gamma)=\frac{1}{2}\sup_{A\in\bD} \tr (\Gamma A).
\]
Indeed, this follows from~\cite{DenisHuPeng.2010} with an additional density argument (see, e.g., \cite[Remark~3.6]{DolinskyNutzSoner.11}).
The main result of this section states our main assumptions are satisfied for the sets $\cP(s,\bomega):=\cP_\bD$; to wit, in this special case, there is no dependence on $s$ or $\bomega$. The result  entails that we can extend the conditional $G$-expectation to upper semianalytic functions and
to stopping times. (The extension is, of course, not unique; cf.\ Section~\ref{se:counterex}.)

\begin{proposition}\label{pr:assumptionSatisfiedForGexp}
  The set $\cP_\bD$ satisfies Assumption~\ref{as:invarianceAndPasting}.
\end{proposition}

This proposition is a special case of Theorem~\ref{th:assumptionSatisfiedForRandomGexp} below.
Nevertheless,  as the corresponding proof in the next section is significantly more involved, we state separately a simple argument for Assumption~\ref{as:invarianceAndPasting}(i). It depends not only on $\bD$ being deterministic, but also on its convexity and compactness.

\begin{lemma}\label{le:GexpMeasClosed}
  The set $\cP_\bD\subseteq \fPO$ is closed for the topology of weak convergence.
\end{lemma}

\begin{proof}
  Let $(P_n)$ be a sequence in $\cP_\bD$ converging weakly to $P\in \fPO$; we need to show that $P\in\fMa$ and that $d\br{B}_t/dt \in \bD$ holds $P\times dt$-a.e. To this end, it suffices to consider a fixed, finite time interval $[0,T]$.

  As $\bD$ is bounded, the Burkholder-Davis-Gundy inequalities yield that there is a constant $C_T$ such that
  \begin{equation}\label{eq:BDG}
    E^{P'} \bigg[\sup_{t\leq T} |B_t|^{4}\bigg] \leq C_T
  \end{equation}
  for all $P'\in\cP$. If $0\leq s\leq t\leq T$ and $f$ is any $\cF_s$-measurable bounded continuous function, it follows that
  \[
    E^P[(B^{(i)}_t-B^{(i)}_s)f]=\lim_n E^{P_n}[(B^{(i)}_t-B^{(i)}_s)f]=0
  \]
  for each component $B^{(i)}$ of $B$; that is, $B$ is a martingale under $P$.

  To see that $d\br{B}_t\ll dt$ $P$-a.s.\ and $d\br{B}_t/dt \in \bD$ $P\times dt\mbox{-a.e.}$, we use an argument similar to a proof in \cite{DolinskyNutzSoner.11}. Given $\Gamma\in\R^{d\times d}$, the separating hyperplane theorem implies that
  \begin{equation}\label{eq:separatingHyperplane}
   \Gamma\in \bD \q\mbox{if and only if}\q \ell(\Gamma)\leq C^\ell:=\sup_{A\in \bD} \ell(A)\quad \mbox{for all}\quad \ell \in (\R^{d\times d})^*,
  \end{equation}
  where $(\R^{d\times d})^*$ is the set of all linear functionals $\ell:\R^{d\times d}\to \R$.
  Now let $\ell \in (\R^{d\times d})^*$, fix $0\leq s<t\leq T$ and set $\Delta_{s,t}B:=B_t-B_s$.
  Let $f\geq0$ be an $\cF_s$-measurable bounded continuous function.
  For each $n$, $B$ is a square-integrable $P_n$-martingale and hence
  \begin{equation}\label{eq:sqIntMart}
    E^{P_n} [(\Delta_{s,t}B)(\Delta_{s,t}B)'|\cF_s]
    =E^{P_n} [B_tB_t' - B_sB_s'|\cF_s]=E^{P_n}[\br{B}_t-\br{B}_s |\cF_s].
  \end{equation}
  Using the convexity of $\bD$, we have $\br{B}_t-\br{B}_s\in (t-s)\bD$ $P_n$-a.s.\ and hence
  \[
    E^{P_n} \big[\ell\big((\Delta_{s,t}B)(\Delta_{s,t}B)'\big)f\big] \leq E^{P_n} \big[C^\ell(t-s)f\big]
  \]
  by~\eqref{eq:separatingHyperplane}. Recalling~\eqref{eq:BDG} and passing to the limit, the same holds with $P_n$ replaced by $P$. We use~\eqref{eq:sqIntMart} for $P$ to deduce that
  \begin{equation}\label{eq:proofClosednessSimple}
    E^P \big[\ell\big(\br{B}_t-\br{B}_s\big)f\big]
    \leq E^P \big[C^\ell (t-s)f\big].
  \end{equation}
  By approximation, this extends to functions $f$ that are $\cF_s$-measurable but not necessarily continuous. It follows that
  if $H\geq0$ is a bounded, measurable and adapted process, then
  \begin{equation}\label{eq:proofClosedness}
    E^P\bigg[\int_0^T H_t\,\ell(d\br{B}_t)\bigg] \leq E^P\bigg[\int_0^T H_t\,C^\ell\,dt\bigg].
  \end{equation}
  Indeed, if $H$ is a step function of the form $H=\sum\1_{(t_i,t_{i+1}]}f_{t_i}$, this is immediate from \eqref{eq:proofClosednessSimple}. By direct approximation, \eqref{eq:proofClosedness} then holds when $H$ has left-continuous paths. To obtain the claim when $H$ is general, let $A'$ be the increasing process obtained by adding the total variation processes of the components of $\br{B}$ and let $A_t=A'_t+t$. Then
  \[
   H^n_t=\frac{1}{A_t-A_{(t-1/n)\vee0}}\int_{(t-1/n)\vee0}^t H_u\,dA_u,\quad t>0
  \]
  defines a bounded nonnegative process with $P$-a.s.\ continuous paths and $H^n(\omega)\to H(\omega)$ in $L^1(dA(\omega))$ for $P$-a.e.\ $\omega\in\Omega$. Thus, we can apply~\eqref{eq:proofClosedness} to $H^n$ and pass to the limit as $n\to\infty$.

  Since $\ell\in (\R^{d\times d})^*$ was arbitrary, \eqref{eq:proofClosedness} implies that
  $d\br{B}_t\ll dt$ $P$-a.s. Moreover, it follows that $\ell(d\br{B}_t/dt) \leq C^\ell$ $P\times dt$-a.e.\ and thus $d\br{B}_t/dt\in\bD$ $P\times dt$-a.e.\ by~\eqref{eq:separatingHyperplane}.
\end{proof}

\section{Application to Random $G$-Expectations}\label{se:randomG}

In this section, we consider an extension of the $G$-expectation, first introduced in \cite{Nutz.10Gexp}, where the set $\bD$ of volatility matrices is allowed to be time-dependent and random. Recalling the formula $G(\Gamma)=\sup_{A\in\bD} \tr (\Gamma A)/2$, this corresponds to a ``random $G$''.
Among other improvements, we shall remove completely the uniform continuity assumption that had to be imposed on $\bD$ in \cite{Nutz.10Gexp}.

We consider a set-valued process $\bD: \Omega \times \R_+\to
2^{\R^{d\times d}}$; i.e., $\bD_t(\omega)$ is a set of matrices for each
$(t,\omega)\in \R_+\times \Omega$.  We assume throughout this
section that $\bD$ is progressively measurable in the sense of graph-measurability.

\begin{assumption}
\label{aspt:rsetmeas}
For every $t\in\R_+$,
\[
	\big\{(s,\omega,A)\in [0,t]\times \Omega\times\R^{d\times d}:
	A\in\bD_s(\omega)\big\}\in
	\cB([0,t])\otimes\cF_t\otimes\cB(\R^{d\times d}),
\]
where $\cB([0,t])$ and $\cB(\R^{d\times d})$ denote the Borel $\sigma$-fields of
$[0,t]$ and $\R^{d\times d}$.
\end{assumption}

In particular, $\bD_t(\omega)$ depends only on the
restriction of $\omega$ to $[0,t]$.  In contrast to the special
case considered in the previous section, $\bD_t(\omega)$ must only be a
Borel set: it need not be bounded, closed, or convex.

\begin{remark}
The notion of measurability needed here is very weak.
It easily implies that if $A$ is a progressively measurable
$\R^{d\times d}$-valued process, then the set
$\{(\omega,t):A_t(\omega)\in\bD_t(\omega)\}$ is a progressively measurable
subset of $\R_+\times\Omega$, which is the main property we need in
the sequel.

A different notion of measurability for \emph{closed} set-valued processes
is the requirement that for every closed set $K\subseteq \R^{d\times d}$,
the lower inverse image $\{(t,\omega):\, \bD_t(\omega)\cap
K\neq\varnothing\}$ is a (progressively) measurable subset of
$\R_+\times \Omega$.  This implies Assumption \ref{aspt:rsetmeas};
cf.\ \cite[Thm.\,1E]{Rockafellar.76}.  However, our setting is more
general as it does not require the sets $\bD_t(\omega)$ to be closed.
\end{remark}

Given $(s,\bomega)\in\R_+\times\Omega$, we define $\cP_\bD(s,\bomega)$ to be the collection of all $P\in\fMa$ such that
\[
  \frac{d\br{B}^P_u}{du}(\omega) \in \bD^{s,\bomega}_{u+s}(\omega):=\bD_{u+s}(\bomega\otimes_s\omega) \quad\mbox{for } du\times P\mbox{-a.e.}\; (u,\omega) \in \R_+\times \Omega.
\]
We set $\cP_\bD=\cP_\bD(0,\bomega)$ as this collection does not depend on $\bomega$. We can then define the sublinear expectation
\[
  \cE_0^\bD(\xi):=\sup_{P\in \cP_\bD} E^P[\xi].
\]
When $\bD$ is compact, convex, deterministic and constant in time, we
recover the setup of the previous section.  The main result of the present section is that our
key assumptions are satisfied for the sets $\cP_\bD(s,\bomega)$. We
recall that $\cP_\bD(\tau,\bomega):=\cP_\bD(\tau(\bomega),\bomega)$ when
$\tau$ is a stopping time.

\begin{theorem}\label{th:assumptionSatisfiedForRandomGexp}
  The sets $\cP_\bD(\tau,\bomega)$, where $\tau$ is a (finite) stopping time and $\bomega\in\Omega$, satisfy Assumption~\ref{as:invarianceAndPasting}.
\end{theorem}

We state the proof as a sequence of lemmata. We shall use several times the following observation:
Given $P\in\fPO$, we have
$P\in\fM$ if and only if for each $1\leq i\leq d$ and $n\geq1$, the $i$th component $B^{(i)}$ of $B$ stopped at $\tau_n$,
\begin{equation}\label{eq:Bstopped}
  Y^{(i,n)}=B^{(i)}_{\cdot\wedge \tau_n},\quad \tau_n=\inf\{u\geq 0:\,|B_u|\geq n\},
\end{equation}
is a martingale under $P$.

We start by recalling (cf.\ \cite{SonerTouziZhang.2010dual}) that
using integration by parts and the pathwise stochastic integration of
Bichteler~\cite[Theorem~7.14]{Bichteler.81}, we can define a progressively
measurable, $\overline{\R}^{\,d\times d}$-valued process $\br{B}$
  such that
  \[
    \br{B}=\br{B}^P\quad P\as\quad\mbox{for all}\quad P\in\fM.
  \]
In particular, $\br{B}$ is continuous and of finite variation $P$-a.s.\
for all $P\in\fM$.

\begin{lemma}\label{le:locMartMeasBorel}
  The set $\fMa\subseteq \fPO$ is Borel-measurable.
\end{lemma}

\begin{proof}
  \emph{Step~1.}
  We first show that $\fM\subseteq \fPO$ is Borel-measurable.
  Let $Y^{(i,n)}$ be a component of the stopped canonical process as in~\eqref{eq:Bstopped},
  and let $(A^u_m)_{m\geq1}$ be an intersection-stable,
  countable generator of $\cF_u$ for $u\geq0$. Then
  \[
    \fM = \bigcap_{i,m,n,u,v} \big\{P\in\fPO:\, E^P[(Y^{(i,n)}_v-Y^{(i,n)}_u)\1_{A^u_m}]=0\big\},
  \]
  where the intersection is taken over all integers $1\leq i\leq d$ and $m,n\geq1$, as well as all rationals $0\leq u\leq v$. Since the evaluation $P\mapsto E^P[f]$ is Borel-measurable for any bounded Borel-measurable function $f$ (c.f.\ \cite[Prop.\,7.25, p.\,133]{BertsekasShreve.78}), this representation entails that
  $\fM$ is Borel-measurable.

  \medskip
  \noindent\emph{Step~2.}
We now show that $\fMa\subseteq \fPO$ is
Borel-measurable. In terms of the process $\br{B}$ defined above, we
have
  \[
    \fMa=\{P\in\fM:\, \br{B}\mbox{ is absolutely continuous }P\as\}.
  \]
We construct a measurable version of the absolutely continuous part of $\br{B}$ as follows. For $n,k\geq 0$, let $A_n^k=(k2^{-n},(k+1)2^{-n}]$.
If $\cA_n$ is the $\sigma$-field generated by
$(A_n^k)_{k\geq 0}$, then $\sigma(\cup_n \cA_n)$ is the Borel
$\sigma$-field $\cB(\R_+)$.  Let
  \[
    \varphi^n_t(\omega)=\sum_{k\geq 0} \1_{A^k_n}(t)
\frac{\br{B}_{(k+1)2^{-n}}(\omega)-\br{B}_{k2^{-n}}(\omega)}{2^{-n}},\quad
(t,\omega)\in \R_+\times\Omega,
  \]
and define (the limit being taken componentwise)
  \[
    \varphi_t(\omega):=\limsup_{n\to\infty} \varphi^n_t(\omega),\quad
	(t,\omega)\in \R_+\times\Omega.
  \]
As $\br{B}$ has finite variation $P$-a.s.\ for $P\in\fM$, it follows
from the martingale convergence theorem (see the remark following
\cite[Theorem\,V.58, p.\,52]{DellacherieMeyer.82}) that $\varphi$ is
$P$-a.s.\ the density of the absolutely continuous part of $\br{B}$ with
respect to the Lebesgue measure. That is, for $P$-a.e.\ $\omega\in\Omega$
and all $t\in\R_+$,
  \[
    \br{B}_t(\omega)=\psi_t(\omega)+\int_0^t \varphi_s(\omega)\,ds,
  \]
where $\psi(\omega)$ is singular with respect to the
Lebesgue measure.  We deduce that
  \[
	\fMa=\bigg\{P\in\fM:\;
	\br{B}_t = \int_0^t \varphi_s\,ds\;\;P\as
	\mbox{ for all }t\in\mathbb{Q}_+
	\bigg\}.
  \]
As $\br{B}$ and $\varphi$ are Borel-measurable by construction, it follows
that $\fMa$ is Borel-measurable (once more, we use \cite[Prop.\,7.25,
p.\,133]{BertsekasShreve.78}).
\end{proof}

In the sequel, we need a progressively measurable version of the
volatility of $B$; i.e., the time derivative of the quadratic variation.
To this end we define the $\overline{\R}^{\,d\times d}$-valued
process (the limit being taken componentwise)
\begin{equation}\label{eq:hata}
  \hat{a}_t(\omega):=\limsup_{n\to\infty} n\big[\br{B}_t(\omega)-\br{B}_{t-1/n}(\omega)\big],\quad t>0
\end{equation}
with $\hat{a}_0=0$. (We choose and fix some convention to subtract
infinities, say $\infty - \infty = -\infty$). Note that we are taking the limit along the
fixed sequence $1/n$, which ensures that $\hat{a}$ is again progressively measurable.
On the other hand, if $P\in\fMa$, then we know
\emph{a priori} that $\br{B}$ is $P$-a.s.\ absolutely continuous and
therefore $\hat{a}$ is $dt\times P$-a.s.\ finite and equal to the
derivative of $\br{B}$, and $\int \hat{a}_t\,dt= \br{B}$ $P$-a.s.
We will only consider $\hat{a}$ in this setting.

Given a stopping time $\tau$, we shall use the following notation associated with a path $\omega\in\Omega$ and a continuous process $X$, respectively:
\begin{equation}\label{eq:shiftedPath}
  \omega^\tau_\cdot:=\omega_{\cdot+\tau(\omega)}-\omega_{\tau(\omega)},\quad X^\tau_{\cdot}:=X_{\cdot+\tau} - X_\tau\;.
\end{equation}
Of course, $X^\tau$ is not to be confused with the ``stopped process'' that is sometimes denoted the same way.

\begin{lemma}\label{le:randomGmeasuresBorel}
  The graph $\{(P,\omega):\, \omega\in\Omega,\, P\in\cP_\bD(\tau,\omega)\}\subseteq \fPO\times \Omega$ is Borel-measurable for any stopping time $\tau$.
\end{lemma}

\begin{proof}
  Let $A= \{\omega\in\Omega:\hat{a}_u(\omega^\tau) \in
         \bD_{u+\tau(\omega)}(\omega)~du\mbox{-a.e.}\}$.
  Then $A$ is a Borel subset of $\Omega$ by Assumption
  \ref{aspt:rsetmeas} and Fubini's theorem.
  Moreover, if $\bomega,\omega\in\Omega$, then
  $\bomega\otimes_\tau\omega\in A$ if and only if
  \[
    \hat{a}_u(\omega)=\hat{a}_u((\bomega\otimes_\tau\omega)^\tau)\in \bD_{u+\tau(\bomega)}(\bomega\otimes_\tau\omega)\equiv (\bD_{u+\tau})^{\tau,\bomega}(\omega)\quad du\mbox{-a.e.}
  \]
  Hence, given $P\in\fMa$, we have $P\in\cP_{\bD}(\tau,\bomega)$ if and only if
  \[
    P\{\omega\in\Omega: \bomega\otimes_\tau\omega\in A\}=1.
  \]
  Set $f=\1_A$; then
  $P\{\omega\in\Omega: \bomega\otimes_\tau\omega\in A\}=E^P[f^{\tau,\bomega}]$. Since $f$ is Borel-measurable, we have from Step~1 of the proof of Theorem~\ref{th:dpp} that the mapping
  $(P,\bomega)\mapsto E^P[f^{\tau,\bomega}]$ is again Borel-measurable. In view of Lemma~\ref{le:locMartMeasBorel}, it follows that
  \[
    \big\{(P,\bomega):\, \bomega\in\Omega,\, P\in\cP_\bD(\tau,\bomega)\big\}
    =\big\{(P,\bomega)\in\fMa\times\Omega:\,E^P[f^{\tau,\bomega}]=1\big\}
  \]
  is Borel-measurable.
\end{proof}

\begin{lemma}\label{le:localMartMeasInvariant}
  Let $\tau$ be a stopping time and $P\in \fM$. Then $P^{\tau,\omega}\in\fM$ for $P$-a.e.\ $\omega\in \Omega$.
\end{lemma}

\begin{proof}
  For simplicity of notation, we state the proof for the one-dimensional case ($d=1$). Recall the notation~\eqref{eq:shiftedPath}. Given any function $X$ on $\Omega$, we denote by $\widehat{X}$ the function defined by
  \[
    \widehat{X}(\omega):=X(\omega^\tau),\quad \omega\in \Omega.
  \]
  This definition entails that $\widehat{X}^{\tau,\omega}=X$ for any $\omega\in\Omega$, that $\widehat{B_u}=B^\tau_u$ for $u\geq 0$, and that $\widehat{X}$ is $\cF_{u+\tau}$-measurable if $X$ is $\cF_u$-measurable.

  Let $0\leq u\leq v$, $P\in \fM$ and let $f$ be a bounded $\cF_u$-measurable function. Moreover,
  fix $n\geq 1$ and let $\sigma_n=\inf\{u\geq 0:\,|B^\tau_u|\geq n\}$. If $Y:=Y^{(1,n)}$ is defined as in~\eqref{eq:Bstopped}, then
  \begin{align*}
    E^{P^{\tau,\omega}}\big[(Y_v-Y_u)f\big]
    &=E^{P^{\tau,\omega}}\big[(\widehat{Y_v}^{\tau,\omega}-\widehat{Y_u}^{\tau,\omega}){\widehat{f\,}}^{\tau,\omega}\big] \\
    &=E^P\big[(\widehat{Y_v}-\widehat{Y_u})\widehat{f} \,\big|\cF_\tau\big](\omega) \\
    &=E^P\big[\big(B^\tau_{v\wedge \sigma_n} - B^\tau_{u\wedge \sigma_n}\big)\widehat{f} \,\big|\cF_\tau\big](\omega) \\
    &=E^P\big[\big(B_{v\wedge \sigma_n+\tau} - B_{u\wedge\sigma_n+\tau}\big)\widehat{f} \,\big|\cF_\tau\big](\omega) \\
    &=0 \quad \mbox{for $P$-a.e.\ $\omega\in \Omega$.}
  \end{align*}
   This shows that $E^{P^{\tau,\omega}}[Y_v-Y_u|\cF_u]=0$ $P^{\tau,\omega}$-a.s.\ for $P$-a.e.\ $\omega\in \Omega$; i.e., $Y$ is a martingale under $P^{\tau,\omega}$.
\end{proof}

\begin{lemma}\label{le:shiftedVolatilityRandomG}
  Let $\tau$ be a stopping time and let $P\in\fMa$. For $P$-a.e.\ $\omega\in \Omega$, we have $P^{\tau,\omega} \in \fMa$ and
  \[
    \hat{a}_u(\tomega)=(\hat{a}_{u+\tau})^{\tau,\omega}(\tomega) \quad\mbox{for}\quad du\times P^{\tau,\omega}\mbox{-a.e.}\quad (u,\tomega)\in \R_+\times \Omega.
  \]
\end{lemma}

\begin{proof}
  The assertion is quite similar to a result of \cite{SonerTouziZhang.2010dual}. The following holds for fixed $\omega\in\Omega$, up to a $P$-nullset. In Lemma~\ref{le:localMartMeasInvariant}, we have already shown that $P^{\tau,\omega} \in \fM$. We observe that
  \[
    \br{B_{\cdot+\tau}-B_\tau}_u(\omega')=\br{B}_{u+\tau}(\omega') - \br{B}_\tau(\omega')\quad\mbox{for}\quad P\mbox{-a.e.}\;\omega'\in\Omega,
  \]
  which implies that
  \[
    \br{B_{\cdot+\tau}-B_\tau}_u(\omega')=\br{B}_{u+\tau}(\omega') - \br{B}_\tau(\omega')
\mbox{ for }
P^\tau_\omega\mbox{-a.e.}\;\omega'\in\{\omega\otimes_\tau\tomega:\, \tomega\in\Omega\}.
  \]
  Noting that
  \[
    \br{B_{\cdot+\tau}-B_\tau}_u(\omega\otimes_\tau\tomega)=\br{B}_{u}(\tomega)
  \]
  and
  \[
    \br{B}_{u+\tau}(\omega\otimes_\tau\tomega) - \br{B}_\tau(\omega\otimes_\tau\tomega)=(\br{B}_{u+\tau})^{\tau,\omega}(\tomega) - \br{B}_\tau(\omega),
  \]
  we deduce that
  \[
    \br{B}_{u}(\tomega)= (\br{B}_{u+\tau})^{\tau,\omega}(\tomega) - \br{B}_\tau(\omega)
  \]
  for $P^{\tau,\omega}$-a.e.\ $\tomega\in\Omega$. The result follows.
\end{proof}

\begin{lemma}\label{le:localMartMeasStable}
  Let $s\in\R_+$, let $\tau\geq s$ be a stopping time, let $\bomega\in\Omega$ and $\theta:=\tau^{s,\bomega}-s$. Let $P\in\fM$, let $\nu: \Omega \to \fPO$  be an $\cF_\theta$-measurable kernel taking values in $\fM$ $P$-a.s., and let $\bar{P}$ be defined as in~\eqref{eq:defPbar}. Then $\bar{P}\in\fM$.
\end{lemma}

\begin{proof}
  We state the proof for the one-dimensional case $d=1$. Let $n\geq1$ and
  let $Y=Y^{(1,n)}$ be defined as in~\eqref{eq:Bstopped}.

  \medskip
  \noindent \emph{Step~1.}
  Let $\theta\leq \rho\leq \rho'$ be stopping times and let $f$
  be a bounded $\cF_\rho$-measurable function; we show that
  $E^{\bar{P}}[(Y_{\rho'}-Y_\rho)f]=0$. For this, it suffices to show that
  $E^{\bar{P}}[(Y_{\rho'}-Y_\rho)f|\cF_\theta]=0$ $\bar{P}$-a.s.

  Fix $\omega\in\Omega$ such that $\bar{P}^{\theta,\omega}=\nu(\omega)\in\fM$; by Lemma~\ref{le:rcpdOfBarP}, such $\omega$ form a set of $\bar{P}$-measure one. %
  We observe that $M_u=Y_{u+\theta(\omega)}^{\theta,\omega}$, $u\geq0$ defines a martingale under any element of $\fM$. Letting
  \[
    \varrho:=(\rho-\theta)^{\theta,\omega}\quad\mbox{and}\quad {\varrho'}:=({\rho'}-\theta)^{\theta,\omega}
  \]
 and recalling that $\nu(\omega)\in\fM$ and that $f^{\theta,\omega}$ is $\cF_{\varrho}$-measurable, we deduce that
 \begin{align*}
  E^{\bar{P}}[(Y_{\rho'}-Y_\rho)f|\cF_\theta](\omega)
  &=E^{\bar{P}^{\theta,\omega}}\big[\big((Y_{\rho'})^{\theta,\omega}-(Y_\rho)^{\theta,\omega}\big)f^{\theta,\omega}\big] \\
  &= E^{\nu(\omega)}\big[\big((Y^{\theta,\omega})_{\varrho'+\theta(\omega)}-(Y^{\theta,\omega})_{\varrho+\theta(\omega)}\big)f^{\theta,\omega}\big]\\
  &= E^{\nu(\omega)}[(M_{\varrho'}-M_{\varrho})f^{\theta,\omega}]\\
  &=0
 \end{align*}
  for $P$-a.e.\ and $\bar{P}$-a.e.\ $\omega\in\Omega$.
  \pagebreak[1]

  \medskip
  \noindent \emph{Step~2.}
  Fix $0\leq s\leq t$ and let $f$ be a bounded $\cF_s$-measurable function;
  we show that $E^{\bar{P}}[(Y_t-Y_s)f]=0$.
  Indeed, we have the trivial identity
  \begin{align*}
    (Y_t-Y_s)f
    & = (Y_{t\vee\theta}-Y_{s\vee\theta})f\1_{\theta\leq s} + (Y_{t\vee\theta}-Y_\theta)f\1_{s<\theta\leq t}\\
    & \phantom{=\;}+ (Y_\theta-Y_{s\wedge\theta})f\1_{s<\theta\leq t}+ (Y_{t\wedge\theta}-Y_{s\wedge\theta})f\1_{t<\theta}\;.
  \end{align*}
  The $\bar{P}$-expectation of the first two summands vanishes by Step 1, whereas the $\bar{P}$-expectation of the last two summands vanishes because
  $\bar{P}=P$ on $\cF_\theta$ and $P\in\fM$. This completes the proof.
\end{proof}

\begin{lemma}\label{le:restrictedLocMartMeasStable}
  Let $s\in\R_+$, let $\tau\geq s$ be a stopping time, let $\bomega\in\Omega$
and $P\in\cP_\bD(s,\bomega)$. Moreover, let $\theta:=\tau^{s,\bomega}-s$,
let $\nu: \Omega \to \fPO$ be an $\cF_\theta$-measurable kernel such that
$\nu(\omega)\in \cP_\bD(\tau,\bomega\otimes_s\omega)$ for
$P$-a.e.\ $\omega\in\Omega$ and let $\bar{P}$ be defined as
in~\eqref{eq:defPbar}. Then $\bar{P}\in\cP_\bD(s,\bomega)$.
\end{lemma}

\begin{proof}
  Lemma~\ref{le:localMartMeasStable} yields that $\bar{P}\in\fM$. Hence, we need to show that $\br{B}$ is absolutely continuous $\bar{P}$-a.s.\ and that
  \[
    (du\times \bar{P})\big\{(u,\omega)\in [0,\infty)\times \Omega: \hat{a}_u(\omega)\notin \bD_{u+s}^{s,\bomega}(\omega)\big\}=0.
  \]
  Since $\bar{P}=P$ on $\cF_{\theta}$ and $P\in\cP(s,\bomega)$, we
  have that $d\br{B}_u\ll du$ on $[\![0,\theta]\!]$ $\bar{P}$-a.s.\  and
  \[
    \hat{a}_u(\omega)\in \bD_{u+s}^{s,\bomega}(\omega)\quad\mbox{for}\quad du\times \bar{P}\mbox{-a.e.}\quad (u,\omega)\in [\![0,\theta]\!].
  \]
  Therefore, we may focus on showing that $d\br{B}_u\ll du$ on $[\![\theta,\infty[\![$ $\bar{P}$-a.s.\ and
  \[
    A:=\big\{(u,\omega)\in [\![\theta,\infty[\![: \hat{a}_u(\omega)\notin \bD_{u+s}^{s,\bomega}(\omega)\big\}
  \]
  is a $du\times \bar{P}$-nullset. We prove only the second assertion; the proof of the absolute continuity is similar but simpler.

  We first observe that $(\1_A)^{\theta,\omega}$ is the indicator function of the set
  \[
    A^{\theta,\omega}:=\big\{(u,\omega')\in [\![\theta(\omega),\infty[\![: \hat{a}_u^{\theta,\omega}(\omega')\notin \bD_{u+s}^{\tau,\bomega\otimes_s\omega}(\omega')\big\}.
  \]
  Since $\nu(\cdot)=\bar{P}^{\theta,\cdot}$ $P$-a.s.\ by Lemma~\ref{le:rcpdOfBarP}, it follows from Lemma~\ref{le:shiftedVolatilityRandomG}, the identity $\theta(\omega)+s=\tau(\bomega\otimes_s\omega)$, and $\nu(\cdot)\in \cP_\bD(\tau,\bomega\otimes_s\cdot)$ $P$-a.s., that
  \begin{align*}
    \big(du&\times \nu(\omega)\big)(A^{\theta,\omega})\\
    &=\big(du\times \nu(\omega)\big)\big\{(u,\omega')\in [\![\theta(\omega),\infty[\![: \hat{a}_u^{\theta,\omega}(\omega')\notin \bD_{u+s}^{\tau,\bomega\otimes_s\omega}(\omega')\big\}\\
    &=\big(dr\times \nu(\omega)\big)\big\{(r,\omega')\in [\![0,\infty[\![: \hat{a}_r(\omega')\notin \bD_{r+\tau(\bomega\otimes_s\omega)}((\bomega\otimes_s\omega)\otimes_\tau\omega')\big\}\\
    &=0\quad \mbox{for $P$-a.e.\ $\omega\in\Omega$.}
  \end{align*}
   Using Fubini's theorem, we conclude that
  \begin{align*}
    (du\times\bar{P})(A)
    &= \iiint(\1_A)^{\theta,\omega}(u,\omega')\,du \,\nu(d\omega';\omega)\,P(d\omega)\\
    & = \int \big(du\times\nu(\omega)\big)(A^{\theta,\omega}) \,P(d\omega)\\
    & =0
  \end{align*}
  as claimed.
\end{proof}

\begin{proof}[Proof of Theorem~\ref{th:assumptionSatisfiedForRandomGexp}]
  The validity of Assumption~\ref{as:invarianceAndPasting}(i) is a direct consequence of Lemma~\ref{le:randomGmeasuresBorel}, Assumption~\ref{as:invarianceAndPasting}(ii) follows from Lemma~\ref{le:shiftedVolatilityRandomG}, and  Assumption~\ref{as:invarianceAndPasting}(iii) is guaranteed by Lemma~\ref{le:restrictedLocMartMeasStable}.
\end{proof}

\section{Counterexamples}\label{se:counterex}

In previous constructions of the $G$-expectation, the conditional
$G$-expectation $\cE_t=\cE^\bD_t$ is defined (up to polar sets) on the
linear space $\mathbb{L}^1_G$, the completion of $C_b(\Omega)$ under the norm
$\cE_0(|\cdot|)$. This space coincides with the set of functions on
$\Omega$ that are $\cP_\bD$-uniformly integrable and admit a
$\cP_\bD$-quasi-continuous version; c.f.\
\cite[Theorem\,25]{DenisHuPeng.2010}.

Our results constitute a substantial extension
in that our functional $\cE_t$ is defined pathwise and for
all Borel-measurable functions.  The price we pay for this is that our
construction does not guarantee that $\cE_t$ is itself Borel-measurable,
so that we must extend consideration to the larger class of upper
semianalytic functions.  This raises several natural questions:

\begin{enumerate}[topsep=3pt, partopsep=0pt, itemsep=1pt,parsep=2pt]
  \item Is the extension of $\cE_t$ from continuous to Borel functions
  unique?
  \item Is it really necessary to consider non-Borel functions?
  Can we regain Borel-measurability by modifying $\cE_t$ on a polar set?
  \item The upper semianalytic functions do not form a linear
  space.  Is it possible to define $\cE_t$ on a linear space that includes
  all Borel functions?
  \item Does there exist an alternative solution to the aggregation
  problem~\eqref{eq:aggreg} that avoids the limitations of our construction?
\end{enumerate}
We will presently show that the answer to each of these questions is
negative even in the fairly regular setting of $G$-expectations. This
justifies our construction and its limitations.

\subsection{$\cE_t$ Is Not Determined by Continuous Functions}

The following examples illustrate that the extension of the
$G$-expectation from $C_b(\Omega)$ to Borel functions is not unique
(unless $\bD$ is a singleton). This is by no means surprising, but we would like to remark that no esoteric
functions need to be cooked up for this purpose.

\begin{example}
   {\textrm
   In dimension $d=1$, consider the sets  $\bD=\{1,2\}$ and $\bD'=[1,2]$, and let $\cP_\bD$ and $\cP_{\bD'}$ be the corresponding sets of measures as in Section~\ref{se:randomG}.
   Then $\cE_t^\bD$ and $\cE_t^{\bD'}$ coincide on the bounded continuous functions:
   \[
     \sup_{P\in \cP_{\bD}} E^P[\xi^{t,\omega}]=\sup_{P\in \cP_{\bD'}} E^P[\xi^{t,\omega}]\quad\mbox{for all}\quad \xi\in C_b(\Omega).
   \]
This can be seen using the PDE construction in~\cite[Sect.\,3]{DenisHuPeng.2010}, or
   by showing directly that $\cP_{\bD'}$ is the closed convex hull of $\cP_\bD$ in $\fPO$.
Of course, $\cE_t^\bD$ and $\cE_t^{\bD'}$ then also coincide on the
completion $\mathbb{L}^1_G$ of $C_b(\Omega)$ under $\cE_0^{\bD}{(|\cdot|)}$.

On the other hand, $\cE_t^\bD$ and $\cE_t^{\bD'}$ do not coincide on
the set of Borel-measurable functions. For instance, let
$A=\{\int_0^\infty|\hat{a}_u-3/2|\,du=0\}$ be the ``set of paths with
volatility $3/2$''. Then $A$ is Borel-measurable, and we clearly have
$\cE_t^{\bD'}(\1_A)=1$  and $\cE_t^{\bD}(\1_A)=0$ for all $t\geq0$.
   }
\end{example}

\begin{example}
  {\textrm
  Still in dimension $d=1$, consider the sets  $\bD=[1,2)$ and
  $\bD'=[1,2]$. Then $\cP_{\bD'}$ is the weak closure of $\cP_\bD$, so that
    $\cE_t^{\bD}$ and $\cE_t^{\bD'}$ coincide on bounded (quasi-)continuous functions.
   On the other hand, consider the set $A=\{\br{B}_1\ge 2\}$.
   Then $A$ is Borel-measurable, and we have $\cE_0^{\bD'}(\1_A)=1$
  and $\cE_0^{\bD}(\1_A)=0$.

   Recalling that $\br{B}_1$ admits a quasi-continuous version (cf.\
  \cite[Lem.\,2.10]{DenisMartini.06}), this also shows that, even if $\xi$ is quasi-continuous and $C\subseteq\mathbb{R}$
  is a closed set, the  event $\1_{\xi\in C}$ need not be
  quasi-continuous.
 }
\end{example}

Both of the above examples show that the $G$-expectation defined on quasi-continuous functions does not uniquely
determine ``$G$-probabilities'' even of quite reasonable sets.

\subsection{$\cE_t$ Cannot Be Chosen Borel}

The following example shows that the conditional $G$-expectation
$\cE_t(\xi)$ of a bounded, Borel-measurable random variable $\xi$ need not
be Borel-measurable. More generally, it shows that $\cE_t(\xi)$ need not
even admit a Borel-measurable version; i.e., there is no Borel-measurable
$\psi$ such that $\psi=\cE_t(\xi)$ $P$-a.s.\ for all $P\in\cP_\bD$.
Therefore, redefining $\cE_t(\xi)$ on a polar set does not alleviate
the measurability problem. This illustrates the necessity of using
analytic sets.

\begin{example}\label{ex:nonBorel}
  {\textrm
  Consider the set $\bD=[1,2]$ in dimension $d=1$, and let
  $\cE_t$ be the $G$-expectation corresponding to the set of
  measures $\cP_{\bD}$ as defined in Section~\ref{se:Gexp}.
  Choose any analytic set $A\subseteq [1,2]$
  that is not Borel, and a Borel-measurable
  function $f:[1,2]\to[1,2]$ such that $f([1,2])=A$ (the
  existence of $A$ and $f$ is classical, cf.\
  \cite[Cor.\ 8.2.17, Cor.\ 8.2.8, and Thm.\ 8.3.6]{Cohn.80}).
  Let $C\subseteq [1,2]\times[1,2]$ be
  the graph of $f$, and define the random variable
  \[
    \xi=\1_C\big(\br{B}_2-\br{B}_1,\br{B}_1\big).
  \]
  Then clearly $\xi$ is
  Borel-measurable.  On the other hand, let $P_x$ be the law of
  $\sqrt{x}W$, where $W$ is a standard Brownian motion and $x\in[1,2]$.  Then $P_x\in\cP_\bD$ and
  $P_x\{\br{B}_1=x\}=1$ for every $x\in[1,2]$.  Moreover, it is clear that
  for any $P\in\cP_\bD$, we must have $P\{\br{B}_1\in[1,2]\}=1$. Using the definition of $\cE_1$, we obtain that
  \begin{align*}
  	\cE_1(\xi)(\omega)
    &= \sup_{P\in\cP_\bD}E^P\big[\1_C\big(\br{B}_1,\br{B}_1(\omega)\big)\big] \\
    &= \sup_{x\in[1,2]}\1_C\big(x,\br{B}_1(\omega)\big)\\
    &=\1_A(\br{B}_1(\omega)).	
  \end{align*}
  We claim that $\cE_1(\xi)=\1_A(\br{B}_1)$ is not Borel-measurable.
  Indeed, note that
  \[
  	\1_A(x)=\int \cE_1(\xi)(\omega)\,P_x(d\omega)
  \]
  for all $x\in[1,2]$. But $x\mapsto P_x$ is clearly Borel-measurable, and
  acting a Borel kernel on a Borel function necessarily yields a Borel
  function. Therefore, as $A$ was chosen to be non-Borel, we have shown that
  $\cE_1(\xi)$ is non-Borel.

  The above argument also shows that there cannot exist
  Borel-measurable versions of $\cE_1(\xi)$. Indeed, let $\psi$ be any
  version of $\cE_1(\xi)$; that is, $\psi=\cE_1(\xi)$ $P$-a.s.\ for all
  $P\in\cP_\bD$.  Then
  \[
  	\int \psi(\omega)\,P_x(d\omega) =
  	\int \cE_1(\xi)(\omega)\,P_x(d\omega) =
  	\1_A(x)
  \]
  for all $x\in[1,2]$. Therefore, as above, $\psi$ cannot be Borel-measurable.
  }
\end{example}

\begin{remark}
One may wonder how nasty a set $C$ is needed to obtain the conclusion of
Example~\ref{ex:nonBorel}.  A more careful inspection shows that we may
choose $C = C'\setminus(\Q\times\R)$, where $C'$ is a
closed subset of $[1,2]\times[1,2]$; indeed, $A=g(\N^\N)$ for a continuous
function $g$, see \cite[Cor.\,8.2.8]{Cohn.80}, while $\N^\N$ and
$[1,2]\setminus\Q$ are homeomorphic; cf.\ \cite[Prop.\,7.5]{BertsekasShreve.78}.
However, the counterexample fails to hold if $C$ itself is closed,
as the projection of a closed subset of $[1,2]\times[1,2]$ is always
Borel; see \cite[Prop.\ 7.32]{BertsekasShreve.78} for this and related
results.  In particular, while the necessity of considering non-Borel
functions is clearly established, it might still be the case that
$\cE_t(\xi)$ is Borel in many cases of interest.
\end{remark}

\subsection{$\cE_t$ Cannot Be Defined on a Linear Space}

Peng \cite{Peng.07} introduces nonlinear expectations abstractly as
sublinear functionals defined on a linear space of functions.
However, the upper semianalytic functions, while closed under many
natural operations (cf.\ \cite[Lem.\,7.30, p.\,178]{BertsekasShreve.78}),
do not form a linear space.  This is quite natural: since our
nonlinear expectations are defined as suprema, it is not too surprising
that their natural domain of definition is ``one-sided''.

Nonetheless, it is interesting to ask whether it is possible to
meaningfully extend our construction of the conditional $G$-expectations
$\cE_t$ to a linear space that includes all bounded Borel functions. The
following example shows that it is impossible to do so within the usual
axioms of set theory (ZFC).

\begin{example}\label{ex:nonLebesgue}
  {\textrm
  Once more, we fix $\bD=[1,2]$ in dimension $d=1$, and denote by
  $\cE_t(\xi)(\omega)=\sup_{P\in\cP_\bD} E^P[\xi^{t,\omega}]$
  the associated $G$-expectation.  Suppose that $\cE_t:\cH\to\cH$ has been
  defined on some space $\cH$ of random variables.
  We observe that every random variable $\xi\in\cH$ should, at the very
  least, be measurable with respect to the $\cP_{\bD}$-completion
  \[
    \cF^{\cP_\bD}=\bigcap_{P\in\cP_\bD}\cF^P,
  \]
  as this is the minimal requirement to make sense even of the expression
  $\cE_0(\xi)=\sup_{P\in\cP_\bD}E^P[\xi]$.
  Moreover, if $\xi$ is $\cF^{\cP_{\bD}}$-measurable and $\cE_t(\xi)$
  satisfies the representation~\eqref{eq:aggreg}, which is one of the
  main motivations for the constructions in this paper, then $\cE_t(\xi)$
  is \emph{a fortiori} $\cF^{\cP_\bD}$-measurable.

  The following is based on the fact that there exists a model
  (G{\"o}del's constructible universe) of the set theory ZFC
  in which, for some analytic
  set $A\subseteq [1,2]\times\mathbb{R}$, the projection $\pi A^c$
  of the complement $A^c$ on the second coordinate is Lebesgue-nonmeasurable;
  cf.\ \cite[Theorem\,3.11, p.\,873]{KanoveiLyubetskii.03}.
  Within this model, we choose a Borel-measurable
  function $f:[1,2]\to[1,2]\times\mathbb{R}$ such that $f([1,2])=A$, and let
  $C\subseteq [1,2]\times[1,2]\times\mathbb{R}$ be the graph of $f$.  Then,
  we define the Borel-measurable random variable
  \[
  	\xi=\1_C\big(\br{B}_3-\br{B}_2,\br{B}_2-\br{B}_1,\br{B}_1\big).
  \]
  Proceeding as in Example~\ref{ex:nonBorel}, we find that
  \[
  	\cE_2(\xi) = \1_A\big(\br{B}_2-\br{B}_1,\br{B}_1\big)\quad\mbox{and}\quad
  	\cE_1\big({-\cE_2(\xi)}\big) = \1_{\pi A^c}(\br{B}_1)-1.
  \]
  We now show that $\1_{\pi A^c}(\br{B}_1)$ is not $\cF^{\cP_\bD}$-measurable.
  To this end, let $P_x$ be the law of $\sqrt{x}W$, where
  $W$ is a standard Brownian motion, and define $P = \int_1^2 P_x\,dx$; note that
  $P\in\cP_\bD$. We claim that $\1_{\pi A^c}(\br{B}_1)$
  is not $\cF^P$-measurable.  Indeed, suppose to the contrary that
  $\1_{\pi A^c}(\br{B}_1)$ is $\cF^P$-measurable, then
  there exist Borel sets
  \[
    \Lambda_-\subseteq\big\{\br{B}_1\in\pi A^c\big\}\subseteq \Lambda_+
  \]
  such that $P(\Lambda_+\setminus \Lambda_-)=0$.  Therefore, if we define $h_\pm(x) =
  P_x[\Lambda_\pm]$, then
  we have $h_-\le \1_{\pi A^c}\le h_+$ pointwise and
  \[
    \int_1^2 \{h_+(x)-h_-(x)\}\,dx=P(\Lambda_+\setminus \Lambda_-)=0.
  \]
  As $\pi A^c$ is Lebesgue-nonmeasurable, this
  entails a contradiction.

  In conclusion, we have shown that $\cE_1(-\cE_2(\xi))$
  is not $\cF^{\cP_\bD}$-measurable.
  This rules out the possibility that $\cE_t:\cH\to\cH$, where $\cH$
  is a linear space that includes all bounded Borel-measurable
  functions.  Indeed, as $\xi$ is Borel-measurable, this would imply that
  $\xi$, $\cE_2(\xi)$, $\xi'=-\cE_2(\xi)$, and $\cE_1(\xi')$ are all in
  $\cH$, which is impossible as $\cE_1(\xi')$ is not
  $\cF^{\cP_\bD}$-measurable. We remark that, as in
  Example~\ref{ex:nonBorel}, modifying $\cE_t$ on a polar set cannot alter
  this conclusion.
  }
\end{example}

\subsection{Implications to the Aggregation Problem}

We have shown above that our particular construction of the conditional
$G$-expectation $\cE_t$ cannot be restricted to Borel-measurable functions
and cannot be meaningfully extended to a linear space.  However, \emph{a
priori}, we have not excluded the possibility that these shortcomings can
be resolved by an entirely different solution to the aggregation problem
\eqref{eq:aggreg}. We will presently show that this is impossible: the
above counterexamples yield direct implications to any potential
construction of the conditional $G$-expectation that satisfies
\eqref{eq:aggreg}.  We work again in the setting of the previous examples.

\begin{example}
Fix $\bD=[1,2]$ in dimension $d=1$.  In the present example, we suppose
that $\cE_t(\xi)$ is any random variable that satisfies the aggregation
condition~\eqref{eq:aggreg} for $\cP=\cP_\bD$ (that is, we do not assume that $\cE_t(\xi)$ is
constructed as in Theorem \ref{th:dpp}). Our claims are as follows:

\begin{enumerate}[topsep=3pt, partopsep=0pt, itemsep=1pt,parsep=2pt]
\item There exists a bounded Borel-measurable random variable $\xi$
such that every solution $\cE_1(\xi)$ to
the aggregation problem \eqref{eq:aggreg} is non-Borel.
\item It is consistent with ZFC that there exists a bounded
Borel-measurable random variable $\xi$ such that, for any solution
$\xi'=\cE_2(\xi)$ to the aggregation problem \eqref{eq:aggreg}, there
exists no solution to the aggregation problem for $\cE_1(-\xi')$.  In
particular, the aggregation problem \eqref{eq:aggreg} for $\cE_t(\psi)$
may admit no solution even when $\psi$ is universally measurable.
\end{enumerate}
Of course, these claims are direct generalizations of our previous
counterexamples.  However, the present formulation sheds light on the
inherent limitations to constructing sublinear expectations through
aggregation.

The proof of (i) follows directly from Example~\ref{ex:nonBorel}.
Indeed, let $\xi$ be as in Example \ref{ex:nonBorel}.  Then Theorem~\ref{th:dpp} proves the existence of one solution to the aggregation
problem~\eqref{eq:aggreg} for $\cE_1(\xi)$.  Moreover, it is immediate
from~\eqref{eq:aggreg} that any two solutions to the aggregation problem
can differ at most on a polar set.  But we have shown in Example~\ref{ex:nonBorel} that any version of $\cE_1(\xi)$ is non-Borel.  Thus the
claim~(i) is established.

For the proof of (ii), we define
$\xi$ and $A$ as in Example \ref{ex:nonLebesgue}; in particular, the
projection $\pi A^c$ is Lebesgue-nonmeasurable in a suitable model of
ZFC.  Let $\xi'$ be any solution to the aggregation problem
\eqref{eq:aggreg} for $\cE_2(\xi)$.  It follows as above that $\xi'$ and
\[
 \xi''=\1_A\big(\br{B}_2-\br{B}_1,\br{B}_1\big)
\]
differ at most on a polar
set. Note that, in general, if
there exists a solution $\cE_t(\psi)$ to the aggregation problem
\eqref{eq:aggreg} for $\psi$, and if $\psi'$ agrees with $\psi$ up to a
polar set, then $\cE_t(\psi)$ also solves the aggregation problem for
$\psi'$.  Therefore, it suffices to establish that there exists no solution
to the aggregation problem for $\cE_1(-\xi'')$.  In the following, we
suppose that $\cE_1(-\xi'')$ exists, and show that this entails a
contradiction.

Let $P_{x,y}$ be the law of $\sqrt{x}W_{\cdot\wedge 1}+
\sqrt{y}(W_{\cdot\vee 1}-W_1)$, where $W$ is a standard Brownian motion,
and let $P_x=P_{x,x}$.  Then $P_{x,y}\in\cP_\bD$ for every
$x,y\in[1,2]$, while $\br{B}_1=x$ and $\br{B}_2-\br{B}_1=y$ $P_{x,y}$-a.s.
Using \eqref{eq:aggreg}, we have
\begin{align*}
	\cE_1(-\xi'') &\ge {\mathop{\esssup}_{y\in[1,2]}}^{P_x}
	E^{P_{x,y}}[-\xi''|\cF_1] \\
	 &= \sup_{y\in[1,2]}\1_{A^c}(y,x)-1\\
     &=\1_{\pi A^c}(x)-1\quad P_x\mbox{-a.s.}
\end{align*}
for every $x\in[1,2]$.  On the other hand, we have
\begin{align*}
	\cE_1(-\xi'') &= \mathop{\esssup^{P_x}}_{P'\in\cP(1;P_x)}
	E^{P'}[\1_{A^c}(\br{B}_2-\br{B}_1,x)|\cF_1]-1 \\
	&\le
	\sup_{y\in[1,2]}\1_{A^c}(y,x)-1\\
    & =\1_{\pi A^c}(x)-1\quad P_x\mbox{-a.s.}
\end{align*}
for every $x\in[1,2]$.  Therefore, we conclude that
\[
	\cE_1(-\xi'') = \1_{\pi A^c}(x)-1 \quad P_x\mbox{-a.s.}
	\quad \mbox{for all }x\in[1,2].
\]
Define $P=\int_1^2 P_x\,dx$.  Then $P\in\cP_\bD$, and~\eqref{eq:aggreg} implies that $\cE_1(-\xi'')$ is $\cF^P$-measurable.
Therefore, there exist Borel functions $$H_-\le \cE_1(-\xi'')\le H_+$$
such that $E^P[H_+ - H_-]=0$.  Defining the Borel
functions $h_\pm(x)=E^{P_x}[H_\pm]$, we find that
$\int_1^2\{h_+(x)-h_-(x)\}\,dx=0$ and
\[
	h_-(x) \le \1_{\pi A^c}(x)-1 \leq h_+(x)\quad\mbox{for all}\quad x\in[1,2].
\]
As $\pi A^c$ is Lebesgue-nonmeasurable, this entails a
contradiction and we conclude that $\cE_1(-\xi'')$ cannot exist.
\end{example}

\newcommand{\dummy}[1]{}

\end{document}